\documentclass{amsart}
\usepackage{amsfonts,amssymb,amscd,amsmath,enumerate,verbatim,calc}
\usepackage[all]{xy}
\usepackage{hyperref}

\newcommand{\CM}{Cohen-Macaulay }

\newcommand{\wrt}{with respect to}

\newcommand{\n}{\mathfrak{n} }
\newcommand{\m}{\mathfrak{m} }

\newcommand{\g}{\mathfrak{gr} }

\providecommand{\f}{{\mathcal F}}
\providecommand{\D}{{\mathcal D}}
\newcommand{\Z}{\mathbb{Z} }

\newcommand{\QQ}{\mathbb{Q} }

\newcommand{\K}{\mathbb{K}_{\bullet} }

\newcommand{\rt}{\rightarrow}
\newcommand{\lrt}{\longrightarrow}
\newcommand{\ov}{\overline}

\newcommand{\Ass}{\operatorname{Ass}}

\newcommand{\hgt}{\operatorname{height}}

\newcommand{\Supp}{\operatorname{Supp}}
\providecommand\Spec{\text{\rm Spec}}

\newcommand{\projdim}{\operatorname{projdim}}
\newcommand{\injdim}{\operatorname{injdim}}

\newcommand{\Hom}{\operatorname{Hom}}

\newcommand{\Tor}{\operatorname{Tor}}

\theoremstyle{plain}

\newtheorem{theorem}{Theorem}[section]
\newtheorem{corollary}[theorem]{Corollary}
\newtheorem{lemma}[theorem]{Lemma}
\newtheorem{proposition}[theorem]{Proposition}

\theoremstyle{definition}
\newtheorem{definition}[theorem]{Definition}
\newtheorem{s}[theorem]{}
\newtheorem{remark}[theorem]{Remark}
\newtheorem{example}[theorem]{Example}

\theoremstyle{remark}
\newtheorem{claim*}{\it Claim:}

\begin{document}

\title[Graded components local cohomology modules]{Graded components of local cohomology modules over invariant rings-II}
\author{Tony.~J.~Puthenpurakal}
\date{\today}
\address{Department of Mathematics, IIT Bombay, Powai, Mumbai 400 076}
\email{tputhen@math.iitb.ac.in}
 \begin{abstract}
Let $A$ be a regular ring containing a field $K$ of characteristic zero and let $R = A[X_1,\ldots, X_m]$. Consider $R$ as standard
graded with $\deg A = 0$ and $\deg X_i = 1$ for all $i$. 
Let $G$ be a finite subgroup of $GL_m(A)$. Let $G$ act linearly on $R$ fixing $A$. Let $S = R^G$.
In this paper we present a comprehensive study of graded components of
local cohomology modules
$H^i_I(S)$ where $I$ is an \emph{arbitrary} homogeneous ideal in $S$. We prove stronger results when $G \subseteq GL_m(K)$. Some of
our results are new even in the case when $A$ is a field.
\end{abstract}
\maketitle
 %\tableofcontents
%\newpage
\section{Introduction}
Let $R = \bigoplus_{n \geq 0}R_n$ be a Noetherian graded ring. $R$ need not be 
standard graded.  Let $M$ be a finitely generated graded 
$R$-module. Then usually only local cohomology modules of $M$ \wrt \ $R_+ = \bigoplus_{n \geq 1}R_n$ is studied (even when $R$ is standard graded).
Note that if $I$ is an arbitrary homogeneous ideal in $R$ then $H^i_I(R)$ is a graded $R$-module
for all $i \geq 0$. There is in general no previous study of properties of graded components of $H^i_I(R)$ when $I \neq R_+$ (even in the case when $R_0$ is a field).

In an earlier paper \cite{TP2} the author showed that if $A$ is a regular ring containing a field of characteristic zero  and 
if $R = A[X_1,\ldots, X_m]$ is standard graded ( with $\deg A = 0$) then the theory of local cohomology of
$R$ with respect to \emph{arbitrary} homogeneous ideals of $R$ exhibit striking good behavior. In \cite{TPSR1} some of these properties were proved for any Noetherian ring $A$ containing a field of characteristic zero.
We should note that local cohomology modules over regular rings does indeed show good behavior.
For instance see the remarkable papers \cite{HuSh}, \cite{Lyu-1} and \cite{Lyu-2}.

It has been known for some time that local cohomology modules over ring $A^G$, where $A$ is regular containing a field $K$ and $G$ is
a finite subgroup of $Aut(A)$ with $|G|$ invertible in $A$, also has good properties ( see \cite{NB} and  \cite{TWP}). In view of this 
the author with a co-worker investigated properties of local cohomology modules with respect to \emph{arbitrary} homogeneous ideals in the following setup
\cite{TPSR2}:

$A$ is regular  domain containing a field of characteristic zero and $G$ is a finite subgroup of $Aut(A)$. Let $A^G$ be the ring of invariants. Set
$S = A^G[X_1,\ldots, X_m]$ (standard graded). Then if $I$ is a homogeneous  ideal of $S$ then $H^i_I(S)$  exhibits striking good behavior especially when $A^G$ is Gorenstein.

It is natural to investigate graded local cohomology modules in the following case:

\s \label{std} \emph{Standard Assumption:} $A$ is a Noetherian ring containing a field $K$ of characteristic zero. Let $G$ be a finite group with a group homomorphism $\Phi \colon G \rt GL_m(A)$.
 Set $R = A[X_1,\ldots, X_m]$, standard
graded with $\deg A = 0$ and $\deg X_i = 1$ for all $i$. Let $G$ act linearly on $R$ fixing $A$. Set $S = R^G$. 
 Let $I$ be a homogeneous ideal in $S$. Set $M = H^i_I(S)$. 
It is well-known that $M$ is a graded $S$-module. Set $M = \bigoplus_{n \in \Z}M_n$.

\begin{remark}
Although we are primarily interested in the case when $A$ is regular we are able to prove some results under considerably weaker hypotheses.
\end{remark}

\begin{remark}
 An important remark is that $S$ is usually not standard graded. We should note that in the previous three cases considered by the 
 author (namely \cite{TP2},\cite{TPSR1} and \cite{TPSR2} the ring under consideration was standard graded. 
\end{remark}

We first give a summary of the results proved in this paper.

\textbf{I:} \textit{(Vanishing:)} The first result we prove is that vanishing of almost all graded components of $M$ implies vanishing of $M$. More precisely we show
 \begin{theorem}\label{vanish}
(with hypotheses as in \ref{std}).  If $M _n = 0$ for all $|n|  \gg 0$ then  
$M = 0$.
 \end{theorem}

 \textbf{II:} \textit{(Testing Vanishing:)} Let $N = H^j_J(R)$ where $J$ is a homogeneous ideal in $R$. 
 In \cite{TPSR1} it is shown that if $N_0 = N_{-m} = 0$ then $N = 0$. A natural question is whether we can test vanishing of graded
 local cohomology modules over $S$ by testing vanishing in finitely many components.
 We show the following:
 \begin{theorem}\label{vanish-test}
(with hypotheses as in \ref{std}). Also assume $A$ is Cohen-Macaulay. 
    If $M_j = 0$ for all $j \in [-m - (|G|! - 1)^m, (|G|! - 1)^m]$ then $M = 0$.
 \end{theorem}

\s \label{setup-regular} \textit{From now on we assume $A$ is regular 
and that $G$ is a finite subgroup of $GL_m(K)$}. Rest of the hypotheses 
is similar to \ref{std}.

 \textbf{III} \textit{(Infinite generation:)} Recall that each component of $H^m_{S_+}(S)$ is a finitely generated $S$-module, 
 cf., \cite[15.1.5]{BS} (for the standard graded case). We give a sufficient condition for infinite generation of a component of graded local cohomology module over $R$.
\begin{theorem}\label{inf-gen}(with hypotheses as in \ref{setup-regular}). Further assume $A$ is a domain. Assume $I \cap A \neq 0$. 
If $M_t \neq 0$ then
$M_t$ is NOT finitely generated as an $A$-module.
\end{theorem}

 \textbf{IV} \textit{(Injective dimension:)}
Let $E$ be an $A$-module. Let
$\injdim_A E$ denotes the injective dimension of $E$. Also 
$\Supp_A E = \{ P \mid  E_P \neq 0 \ \text{and $P$ is a prime in $A$}\}$ is the support of an $A$-module $E$. 
By $\dim_A E $ we mean the dimension of $\Supp_A E$ as a subspace of $\Spec(A)$.
\begin{theorem}\label{inj-dim}(with hypotheses as in \ref{setup-regular}).  If $M_t \neq 0$ then
$\injdim M_t \leq \dim M_t$.
\end{theorem}
\textbf{V} \textit{(Bass numbers:)}
 Over $A[X_1,\ldots, X_m]$ bass numbers of graded components of local cohomology modules exhibited a dichotomy: either
 they were infinite for all $n$ or finite for all $n$ see \cite[1.8]{TP2}. For invariant rings we have the following:

\begin{theorem}\label{bass}(with hypotheses as in \ref{setup-regular}.) 
 Let $M = H^i_I(S) = \bigoplus_{n \in \Z}M_n$ and let $P$ be a prime ideal of $A$ and let 
 $j \geq 0$. Let $\mathcal{O}(l)  = l + |G|!\Z$ be a coset of $|G|!\Z$. Then
if $\mu_j(P, M_{n_0})$ is finite for some $n_0 \in \mathcal{O}(l)$ then it is finite for 
all $n \in \mathcal{O}(l)$.
\end{theorem}
 
 \textbf{VI} (\textit{ Growth of Bass numbers}). Fix $j \geq 0$. Let $P$ be a prime ideal in $A$ such that $\mu_j(P, H^i_I(S)_n)$ is finite for all 
 $n \in \mathcal{O}(l)$ where $\mathcal{O}(l)$ is as in Theorem \ref{bass}.
 $n \in \Z$.  Set $n = l + et$. We may ask about the growth of the function $n \mapsto \mu_j(P, H^i_I(S)_n)$ as $t \rt -\infty$ and when $t \rt + \infty$. We prove
  
\begin{theorem}\label{basspoly-intro}(with hypotheses as in \ref{bass}.)
If $\mu_j(P, M_n)< \infty$ for all $n \in \mathcal{O}(l)$, then there exists $\alpha(X)$, $\beta(X) \in Q[X]$ of degree $\leq m-1$ such that 
\begin{enumerate}[\rm (1)]
\item $\mu_j(P, \f(S)_n)=\alpha(t)$ for all $t \gg 0$ where $n = l + |G|!t$.
\item $\mu_j(P, \f(S)_n)= \beta(t)$ for all $t \ll 0$ where $n = l + |G|!t$.
\end{enumerate}
\end{theorem}

\textbf{VII}(\textit{Associate Primes})
Let $E= \bigoplus_{n \in \Z}E_n$ be a graded module over a graded ring $S= \bigoplus_{n \geq 0}S_n$ (not necessarily standaed graded). For associated primes we ask 
\begin{enumerate}[\rm 1)]
\item Is $\bigcup_{n \in \Z}\Ass_{S_0} E_n$ is finite?
\item Does there exists $v >0$  and $n_0, n_0'$ such that if $\Theta_l = l + v \Z$ for some $0 \leq l \leq v-1$ then 
\begin{enumerate}[\rm i)]
\item $\Ass_{S_0} E_n = \Ass_{S_0} E_{n_0}$ for all $n \in \Theta_l$ and $n \geq n_0$?
\item $\Ass_{S_0} E_n = \Ass_{S_0} E_{n'_0}$ for all $n \in \Theta_l$ and $n \leq n'_0$?
\end{enumerate}
\end{enumerate}
We call $(2)$ {\it periodic stability} of Associated primes of graded components  of $E$ \wrt \ $S_0$.
\begin{theorem}\label{ass-primes}(with hypotheses as in \ref{setup-regular}.)
Further assume either $(A, \m)$ is a regular local ring or is a smooth affine variety.
Let $I$ be a homogeneous ideal in $S$ and set $M = H^i_I(S) = \bigoplus_{n \in \Z}M_n$
\begin{enumerate}[\rm 1)]
\item $\bigcup_{i \in \Z} \Ass_A M_i$ is a finite set.
\item There exists periodic stability of associated primes of graded components of   $M$ \wrt \ $A$.
\end{enumerate}
\end{theorem}

We now discuss:\\
\emph{The special case when $A$ is a field  $K$ (of characteristic zero):} In this special case all our results except Theorem's  
\ref{inf-gen}, \ref{inj-dim} and \ref{ass-primes}(1) yield non-trivial, previously unknown results about graded components of local cohomology modules of homogenous ideals
in $K[X_1,\ldots, X_m]^G$. Clearly Theorem's \ref{vanish} and \ref{vanish-test} are intetersting in this case.
Regarding bass numbers note that only $0$ is the only prime ideal
and $\mu_0(0, M_n) = \dim_K M_n$. Our result Theorem \ref{bass} states that in a coset of $|G|! \mathbb{Z}$ all these
dimensions are finite or infinite. Furthermore if they are finite it coincides with asymptotically with a polynomial. Theorem
\ref{ass-primes}(2)is an easy consequence of these results.

We now describe in brief the contents of this paper. In section two we discuss some preliminary results that we need.
In section three we prove Theorem \ref{vanish} while Theorem \ref{vanish-test} is proved in section four. In section five we prove Theorem \ref{inf-gen} and \ref{inj-dim}. In section six we prove Theorem's \ref{bass} and \ref{basspoly-intro}.  Finally in section seven  we prove Theorem \ref{ass-primes}.
\begin{remark}
We will prove all our results more generally for $\f(S)$ where $\f= H_{I_1}^{i_1} \circ \cdots \circ H_{I_s}^{i_s}$ and $I_1, \ldots, I_s$ are homogeneous ideals of $S$.
This generality is needed since if we investigate Bass numbers of $H^i_I(S)_n$
\wrt \ prime $P$ of $A$ then we have to investigate graded components of  $H^j_{PS}(H^i_I(S))$.
\end{remark}

\section{preliminary results}
In this section we discuss a few preliminary results that we need.

\s\emph{Setup}:\label{setup} Throughout $A$ is a commutative Noetherian ring containing a field of characteristic zero, $G$ is a finite group and $\Phi: G \to Gl_m(A)$ is a group homomorphism. Let $R= A[X_1, \ldots, X_m]$ be standard graded with $\deg A=0$ and $\deg X_i=1$ for all $i$. Let $G$ act's linearly on $R$ fixing $A$ and $S= A[X_1, \ldots, X_m]^G$ be the ring of invariants of $G$. Throughout we will work only with homogeneous ideals in $R$ and $S$. 

\begin{lemma}[with the hypothesis as in setup \ref{setup}]\label{flat}
Let $A \to B$ a flat extension. Note $Gl_m(A) \to Gl_m(B)$ is a group homomorphism. So we have a natural group homomorphism $G \to Gl_m(B)$. Then $B \otimes_A(A[X_1, \ldots, X_m]^G) \cong (B[X_1, \ldots, X_m])^G$ as rings. 
\end{lemma}

\begin{proof}
Note $B \otimes_A (A[\underline{\bf X}]^G) \subseteq (B[\underline{\bf X}])^G$. Set $F= A[X_1, \ldots, X_m]_i$. Then we have a split exact sequence of free $A$-modules 
\[
\xymatrix{
0 \ar[r]& F^G \ar@{^{(}->}[r]^i & F  \ar@/^1pc/[l]^\rho \ar[r] & F/F^G \ar[r] & 0}\]
where $\rho: F \to F^G$ is the Reynolds operator. Then \[0 \to B \otimes F^G \to B \otimes F \to B \otimes F/F^G \to 0\]
also split. Note if $b \in B$ and $ f \in F$ then for any $\sigma \in G$ we have $\sigma (b \otimes f)= b \otimes \sigma(f)$.
Thus $\rho(b \otimes f)= b \otimes \rho(f)$. 
If $\xi \in (B \otimes F)^G$ then $\xi=b_1 \otimes f_1 + \cdots+ b_s \otimes f_s$ (say) where
$b_i \in B$ and $f_i \in F$. Hence $\xi=\rho(\xi)= \sum b_i \otimes \rho(f_i) \in B \otimes F^G$.
Thus $(B \otimes F)^G=B \otimes F^G$. So the result follows.  
\end{proof}

\s \label{use-flat} We will use Lemma \ref{flat} in the following instances.

1. $W$ is a multiplicatively closed subset of $A$. Then $B = W^{-1}A$.

2. $(A,\m)$ is Noetherian local and $B = \widehat{A}$ the completion of $A$ \wrt \ $\m$.

3. If $A$ only contains a countable field $K$ then we set $B = A[[X]]_X$.
Note $B$ is a faithfully flat extension of 
$A$; for instance see \cite[3.2]{P3}.  Also $B$ contains $K((X)) = K[[X]]_X$ which is uncountable.

The following result is definitely known to the experts. We give a proof for convenience of the reader. I thank J. K. Verma
for indicating a proof.
\begin{proposition}\label{hsop}
Let $k$ be an  field of characteristic zero  and let $G$ be a finite subgroup of $GL_m(k)$. Let $R = k[X_1, \ldots, X_m]$ and let 
$G$ act linearly on $R$. Set $S= R^G$. Then there exists a homogeneous system of parameters $f_1,\ldots, f_m$ of $S$ with $\deg f_i = |G|!$
for all $i$.
\end{proposition}
\begin{proof}
By a result of E. Noether we get that $S = k[u_1,\ldots, u_l]$ for some $l$ with $\deg u_i \leq |G|$;   cf. \cite[2.4.2]{LS}.
Set $s_i = |G|!/\deg u_i$ 
and $T = k[ u_1^{s_1}, \ldots, u_l^{s_l}]$. We note that $\deg u_j^{s_j} = |G|!$ for all $j = 1, \ldots, l$. Clearly $m$ general
linear combination of $u_j^{s_j}$, say $f_1,\ldots, f_m$, will form a homogeneous system of parameters of $T$. As $S$ is a finite extension of $T$ we get
that $f_1, \ldots, f_m$ is a homogeneous system of parameters of $S$. Note $\deg f_j = |G|!$ for all $j$.
\end{proof}
\begin{proposition}[with the hypothesis as in setup \ref{setup}]\label{grade}
Also assume  that $(A,\m)$ is Cohen-Macaulay. Then there exists homogeneous   $f_1, \ldots, f_m$ in $S$ such that
\begin{enumerate}[\rm 1)]
\item $f_1,\ldots, f_m$ forms a regular sequence in $R$ and hence in $S$.
\item $\deg f_i = |G|!$ for all $i$.
\item 
\[
(R/(\underline{f}))_j = \begin{cases}
                        0  & \text{if} \  j > (|G|! - 1)^m \\
                        \neq 0 & \text{for } \ j = 0,1,\ldots, (|G|! - 1)^m.
                         \end{cases}
\]
\item
$(R/(\underline{f}))$ is a finite free $A$-module.
\end{enumerate}
\end{proposition}

\begin{proof}
Set $k = A/\m$. Note $k$ has characteristic zero. Note we have a natural group homomorphism $GL_m(A) \rt GL_m(k)$ and so we have a
group homomorphism $G \rt GL_m(k)$. Set $U = k[X_1,\ldots, X_m]$ and let $G$ act linearly on $U$. Set $V = U^G$. We have
obvious graded ring homomorphisms $\alpha \colon R \rt U$ and $\beta \colon S \rt V$. Clearly $\alpha$ is surjective.

\textit{Claim:} $\beta$ is surjective.\\
Let $\xi \in V$ be homogeneous. As $\alpha$ is surjective  we can choose $\theta \in R$ homogeneous with $\ov{\theta} = \xi$ and  $\deg \theta = \deg \xi$.
We note that if $\rho^R_S \colon R \rt S$ and $\rho_V^U \colon U \rt V$ are the Reynolds operators then
\[
\ov{\rho^R_S(\theta)} = \rho^U_V(\ov{\theta}) = \rho^U_V(\xi) = \xi.
\]
Notice $\rho^R_S(\theta) \in S$ and its degree equals $\deg \theta = \deg \xi$. Thus $\beta$ is surjective. 

1) and 2): Choose $g_1,\ldots, g_m$ a homogeneous system of parameters of $V$ with $\deg g_i = |G|!$ for all $i$. As $U$ is a finite $V$-module
we note that $g_1,\ldots, g_m$ also form a homogeneous system of parameters of $U$. 
As $\beta$ is surjective we may choose homogeneous $f_i \in S$ with $\ov{f_i} = g_i$ for all $i$.  
It follows that
$\sqrt{\m R + (f_1,\ldots, f_m)R}  = \mathcal{M}$, the maximal homogeneous ideal of $R$. If $y_1,\ldots, y_d$ is a system of parameters
of $A$ it follows that $y_1,\ldots, y_d, f_1,\ldots, f_m$ forms a homogeneous system of parameters of $R$. As $R$ is $*$-local \CM \ ring we get
that it is also an $R$-regular sequence. It follows that $f_1,\ldots, f_m$ is a regular sequence in $R$. By \cite[6.4.4]{BH}
we get that $\underline{f}$ is also a regular sequence in $S$.

3) Set $c = |G|!$.  We note that $g_1,\ldots, g_m$ is a $U$-regular sequence. By looking at Hilbert-Series we conclude
\[
(U/\underline{g})_j = \begin{cases}
                        0  & \text{if} \  j > (c-1)^m \\
                        \neq 0 & \text{for } \ j = 0,1,\ldots, (c-1)^m.
                         \end{cases}
\]
Using Nakayama's lemma the result follows. 

4) Note $R/(\underline{f})$ is a finitely generated  \CM $A$-module.
By taking a projective resolution of $R/(\underline{f})$ (as $R$-modules) we get that
$\projdim_A (R/(\underline{f}))_i < \infty $ for all $i$. So $\projdim R/(\underline{f})$ is finite.
Also $\dim_A R/(\underline{f}) = \dim A$ (as $\left(R/(\underline{f})\right)_0 = A$. 
So $R/(\underline{f})$ is a free $A$-module.

\end{proof}

\s Let $A$ be a Noetherian ring containing a field of characteristic zero and let $R = A[X_1,\ldots, X_m]$. Let $D = A_m(A)$ be the $m^{th}$-Weyl algebra over $A$
 For the notion of generlized Eulerian $D$-modules see \cite{TP2}.
 \section{Skew Group rings and group actions on rings of differential operators}
 
In this section we describe some preliminary results on skew group rings that we need. We also describe conditions when group action over a ring $R$  can be extended to certain ring of differential operators over $R$.  
 
\s \label{Setup} In this section $A$ is a ring (not necessarily commutative) and $G$ is a finite subgroup of $Aut(A)$; the group of automorphisms of $A$. We assume that $|G|$ is invertible in $A$. 
 \s Recall that the skew group ring of $A$ \wrt \ $G$ is 
$$A*G = \{ \sum_{\sigma \in G} a_\sigma \sigma \mid a_{\sigma} \in A \ \text{for all} \ \sigma \},$$
with multiplication defined as
\[
(a_{\sigma}\sigma)(a_\tau \tau) = a_{\sigma}\sigma(a_{\tau})\sigma \tau.
\]  

\begin{remark}
An $A*G$ module $M$ is precisely an $A$-module on which $G$ acts such that for all $\sigma \in G$,
\[
\sigma(am) = \sigma(a)\sigma(m) \quad \text{for all} \ a \in A \ \text{and} \ m \in M.
\] 
\end{remark}

\begin{definition}
Let $M$ be an $A*G$-module. Then
\[
M^G = \{ m \in M \mid \sigma(m) = m \ \text{for all} \ \sigma \in G \}.
\]
\end{definition}
In particular set $A^G$ to be the ring of invariants of $G$. Clearly $M^G$ is an $A^G$-module.

For further properties of ring of invariants pertinent to our investigation see \cite{TWP}.

\s \label{graded-action} (\textit{Group action on Graded rings}:) If $A$ is graded, say $A= \bigoplus_{n \in \Z}A_n$  then we assume further that that $\sigma(A_n) \subseteq A_n$ for all $n$. In this case $A*G$ is graded with grading defined by, for $a \in A_n$ and $\sigma \in G$ set $\deg a\sigma = n$.
If $A$ is graded then we consider graded $A*G$ modules, i.e., an $A*G$ module 
$M$ which is graded say $M = \bigoplus_{n \in \Z}M_n$ such that $\sigma(M_n) \subseteq M_n$ for all $n$. In this case $M^G$ is a graded $A^G$-module.
Furthermore we note that $G$ acts on $M_n$ and we have an equality
\[
(M_n)^G \cong (M^G)_n \quad \text{for all} \ n \in \Z.
\] 
\s Now assume $A$ is a commutative Noetherian ring. In \cite[Corollary 3.3]{TWP} it is proved that if $I_1, I_2, \cdots, I_r$ are ideals in $A^G$ then
for all $ i_j \geq 0$, where $j = 1,\ldots,r$ the local cohomology module
$H^{i_1}_{I_1A}(H^{i_2}_{I_2A}(\cdots (H^{i_r}_{I_rA}(A))\cdots )$ is an $A*G$-module. Furthermore we have an isomorphism of $A^G$-modules
\[
H^{i_1}_{I_1A}(H^{i_2}_{I_2A}(\cdots (H^{i_r}_{I_rA}(A))\cdots )^G \cong  H^{i_1}_{I_1}(H^{i_2}_{I_2}(\cdots (H^{i_r}_{I_r}(A^G))\cdots ).
\]

\s \label{graded-loc} Now further assume that  $A = \bigoplus_{n \in \Z}A_n$ is graded and $G$ preserves grading as in \ref{graded-action}. Let $I_1,\ldots, I_r$ be graded ideals of $A^G$. Then by proof of  Theorem 3.2 \cite{TWP} it follows that $H^{i_1}_{I_1A}(H^{i_2}_{I_2A}(\cdots (H^{i_r}_{I_rA}(A))\cdots )$ are \emph{graded} $A*G$-module. Furthermore we have \emph{a graded isomorphism:}
of $A^G$-modules
\[
H^{i_1}_{I_1A}(H^{i_2}_{I_2A}(\cdots (H^{i_r}_{I_rA}(A))\cdots )^G \cong  H^{i_1}_{I_1}(H^{i_2}_{I_2}(\cdots (H^{i_r}_{I_r}(A^G))\cdots ).
\]

\s  Let $A$ be a (not necessarily commutative) Noetherian ring containing a field $K$ of characteristic zero. Now assume that $R = A[X_1,\ldots, X_m]$ and that we have a finite group $G$ with a group homomorphism $G \rt GL_m(A)$. We let $G$ act linearly on $R$.
Let $D(R)$ be the ring of $A$-linear differential operators on $R$. Note $D(R) \cong A_m(A)$ the $m^{th}$-Weyl algebra over $A$. We also note that
$A_m(A) \cong A \otimes_K A_m(K)$. We recall a natural action of $G$ on $D(R)$
and then consider the ring of invariants $D(R)^G$. All the assertions regarding $D(R)$ and $D(R)^G$ in this paper follow from \cite[Section 8]{TWP} where it was proved when $A$ is a field of characteristic zero. The same proof's work in general.

\s We first recall the construction of $D(R)$ as a subring of $ S = \Hom_A(R,R)$. The composition of two elements $P, Q$ of $S$ will be denoted as $P\cdot Q$. The commutator of
$P$ and $Q$ is the element 
\[
[P,Q] = P \cdot Q - Q \cdot P.
\]
We have natural inclusion $\eta \colon R \rt S$ where $\eta(r) \colon R \rt R$ is multiplication by $r$.

Set $D_0(R) = R$ viewed as a subring of $S$. For $i \geq 1$ set
\[
D_i(R) = \{ P \in S \mid [P,r] \in D_{i-1}(R) \}.
\]
Elements of $D_i(R)$ are said to be $A$-linear differential operators on $R$ of degree $\leq i$.
Notice $D_{i+1}(R) \supseteq D_i(R)$ for all $i \geq 0$. Set
\[
D(R) = \bigcup_{i \geq 0}D_i(R).
\]
This is the ring of $K$-linear differential operators on $R$. It can be shown that 
$D(R) \cong A_n(K)$.  Set $D(R)_{-1} = 0$. Note that the graded ring
\[
\g(D(R)) = \bigoplus_{i\geq 0}D(R)_i/D(R)_{i-1} = R[\overline{\partial_1},\cdots\overline{\partial_n}],
\]
is isomorphic to the polynomial ring in $n$-variables over $R$.
\s \label{def-action} We define action of $G$ on $D(R)$ as follows. Let $\theta \in D(R)_i$. Let $g \in G$. Define
\begin{align*}
g\theta \colon R &\rt R \\
           r &\rt g\cdot\theta(g^{-1}r). \\
\end{align*}
It can be verified that $g\theta \in D(R)_i$. Thus we have an action of $G$ on $D(R)$. It is easily verified we have a homomorphism  $G \rt Aut(D(R))$.

Let $\Phi \colon G \rt GL_n(A)$ be the original map. 
 It can be verified that the $G$ action on $T=A[\partial_1, \ldots, \partial_m]$
is induced by the  map $\Phi^* \colon G \rt GL_n(A)$ defined by $\Phi^*(g) =  (\Phi(g)^{-1})^t$. Note that $G$ acts linearly on $T$
 
\s By proof of \cite[Proposition 8.6]{TWP} we get that $R$ is a  graded $D(R)*G$ module.
Then by proof of  Theorem 8.8 \cite{TWP} it follows that $H^{i_1}_{I_1A}(H^{i_2}_{I_2A}(\cdots (H^{i_r}_{I_rA}(A))\cdots )$ are \emph{graded} $D(R)*G$-module. Furthermore we have \emph{a graded isomorphism:} of $D(R)^G$-modules
\[
H^{i_1}_{I_1R}(H^{i_2}_{I_2R}(\cdots (H^{i_r}_{I_rR}(R))\cdots )^G \cong  H^{i_1}_{I_1}(H^{i_2}_{I_2}(\cdots (H^{i_r}_{I_r}(S))\cdots ).
\]

\s \label{diff-op-local-1} Now consider the case when $A = K[[Y_1,\ldots, Y_d]]$ where $K$ is a field of characteristic zero. Let $D_K(A) = A<\delta_1,\ldots, \delta_d>$ where $\delta_i = \partial/\partial Y_i$ is the ring $K$-linear differential operators over $A$. We first give an example which shows that if
$G$ is a finite subgroup of $GL_n(A)$ then the $A$-linear action over $R = A[X_1,\ldots, X_d]$ \emph{need not} extend to a $D_K(A)$-linear action over $D_K(A)[X_1,\ldots, X_m]$.
\begin{example}
Let $ A = K[[Y]]$ and let $R = A[X_1, X_2]$. Consider the following subgroup  $G$ of $GL_2(A)$ where 
\[
G = \left\{ 1, \sigma = \begin{pmatrix}
1 & Y \\  0 & -1
\end{pmatrix} \right\}.
\]
Let $B = D_K(A) = A<\delta> $ where $\delta = \partial/\partial Y$. If $G$ extends 
to a $B$-linear action over $B[X_1, X_2]$ then we will have
\[
\sigma (X_2) \delta = \sigma(X_2 \delta) = \sigma (\delta X_2) = \delta \sigma(X_2). 
\]
An easy computation yields a contradiction.
\end{example}

 \s \label{diff-op-local-2}(with hypotheses as in \ref{diff-op-local-1}:) The problem with the above example was that entries 
 in $G$ did not belong to the center of $D_K(A)$. This problem vanishes if we further assume $G$ is a finite  subgroup of
 $GL_m(K)$. In this case it is elementary to see that we have a $D_K(A)$-linear action over $D_K(A)[X_1,\ldots, X_m]$
 which extends the $A$-linear action of $G$ over $R = A[X_1,\ldots, X_m]$. Now let $\D = A_m(D_K(A))$ the $m^{th}$-Weyl algebra
 over $D_K(A)$. Notice $\D \cong D_K(A)\otimes_K A_m(K)$. It is routine to verify that $R$ is a graded $\D*G$-module. Then by a 
 \textit{tedious but routine}  
   computation  it follows that $H^{i_1}_{I_1R}(H^{i_2}_{I_2R}(\cdots (H^{i_r}_{I_rR}(R))\cdots )$ are \emph{graded}
   $\D*G$-module. Furthermore we have \emph{a graded isomorphism} of $\D^G$-modules:
\[
H^{i_1}_{I_1R}(H^{i_2}_{I_2R}(\cdots (H^{i_r}_{I_rR}(R))\cdots )^G \cong  H^{i_1}_{I_1}(H^{i_2}_{I_2}(\cdots (H^{i_r}_{I_r}(S))\cdots ).
\]
\section{Proof of Theorem \ref{vanish}}
 In this section we state and prove the following general result.
 \begin{theorem}\label{vanish-main}[with the hypothesis as in setup \ref{setup}]
Let $I_1, \ldots, I_r$ be homogeneous ideals in $S$. Set $M= H_{I_1}^{i_1}(H_{I_2}^{i_2} \cdots (H_{I_r}^{i_r}(S) \cdots))= \bigoplus_{i \in \Z}M_i$. If $M_{i_0} \neq 0$ for some $i_0$, then $M_i \neq 0$ for infinitely many $i$.
 \end{theorem}
Before proving this let us note that  Theorem \ref{vanish-main} implies Theorem \ref{vanish}.
 
\begin{proof}
Set $N= H_{I_1R}^{i_1}(H_{I_2R}^{i_2} \cdots (H_{I_rR}^{i_r}(R) \cdots))$. Then $N$ is a graded $R*G$-module and $N^G=M$. We note that $N$ is a generalized Eulerian $A_m(A)$-module by \cite[Theorem 3.6]{TP2}. 

{\it Case} 1: $i_0 \geq -m + 1$.

We have a short exact sequence of $A_m(A)$-modules 
\[0 \to \Gamma_{(\underline{\bf X})}(N) \to N \to \overline{N} \to 0.\] Now by \cite[Proposition 4.5]{TPSR1} we 
have $\Gamma_{(\underline{\bf X})}(N)_j=0$ for $j \geq -m+1$. So $\overline{N_j}=N_j$ for $j \geq -m+1$. 
Now $\Ass_R \overline{N}= \{P \in \Ass_R N~|~ P \nsupseteq R_+\}$.  If $A$ contains only a countable field, then set $B= A[[\underline{\bf X}]]_X$. 
Now $B \otimes_A \left(H_{I_1}^{i_1}(H_{I_2}^{i_2} \cdots (H_{I_r}^{i_r}(S) \cdots))\right)= H_{I_1B}^{i_1}(H_{I_2B}^{i_2} \cdots (H_{I_rB}^{i_r}(B \otimes_A S) \cdots)$. 
Also $B \otimes_A S= \left(B[\underline{\bf X}]\right)^G$ by Lemma 
\ref{flat} and $M \otimes_A B= \bigoplus_{i \in \Z} (M_i \otimes_A B)$.
Since $A \to B$ is a faithfully flat extensions so we have $M_j \neq 0$ if and only if $M_j \otimes_A B \neq 0$.
Thus we can assume that $A$ contains a uncountable field, say $K$. Now by \cite[Lemma 5.5]{TPSR1} we have
$N$ is countably generated as an $R$-module. So $\overline{N}$ is countably generated as an $R$-module.
So $\Ass_R \overline{N}$ is a countable set.  Note $L=K[X_1, \ldots, X_m] $ is a subring of $R$. 
Let $\Ass R= \{Q_1, \ldots Q_s\}$. Then $Q_i$'s are graded and $Q_i \nsupseteq R_+$.
Now $L_1 \neq P \cap L_1$ for any $P \in \Ass_R \overline{N}$ and $L_1 \neq Q_j \cap L_1$ for $Q_j \in \Ass R$. 
As $K$ is uncountable there exists some
$$\xi \in L_1 \setminus\left(\bigcup_{P \in \Ass_R \overline{N}}(P \cap L_1) \cup (\bigcup_{Q \in \Ass R} Q \cap L_1)\right).$$
Then $\xi \in R_1$ is a non-zero divisor in $R$ and on $\overline{N}$. As $R$ and $\overline{N} $ are $R*G$-modules so 
$\sigma(\xi)$ is also a non-zero divisor on $R \oplus \overline{N}$ for all $\sigma \in G$. 
Hence $\theta= \prod_{\sigma \in G}\sigma(\xi)$ is also a non-zero divisor on $R \oplus \overline{N}$. 
Note $\theta \in S_+$ and denote $\deg \theta=|\theta|$. Since $N_i=\overline{N_i}$ for $i \geq -m+1$ so we get $N_i \hookrightarrow N_{i+|\theta|}$ for all $i \geq -m+1$. Taking invariant's we get $M_i \hookrightarrow M_{i+|\theta|}$ and the result follows in this case.

{\it Case} 2: $M_{i_0}\neq 0$ for some $i_0 \leq -m$.

We have a short exact sequence \[0 \to \Gamma_{(\underline{\bf \partial})}(N) \to N \to N' \to 0.\] of 
generalized Eulerian $A_m(A)$-modules. Now from \cite[Proposition 4.5]{TPSR1} we have
$\Gamma_{(\underline{\bf \partial})}(N)_j=0$ for $j \leq -m$. So $N_j= N'_j$ for $j \leq -m$. 
Set $T=A[\partial_1, \ldots, \partial_m]$ where $\deg \partial_i=-1$ for all $i$ and $T_{-1}= \bigoplus_{j<0}T_j$.
If $Q \in \Ass_T N'$ then $Q \nsupseteq T_{-1}$.
As before we may assume that $A$ contains uncountable field $K$.
Set $L= K[\partial_1, \ldots, \partial_m]$.  Now we have $\Ass_L N'= \Ass_T N' \cap L$ (as $L \to T$ is a ring homomorphism).
As before there exists some 
$$\xi \in L_{-1} \setminus \left( \bigcup_{P \in \Ass_T N'}P \cap L_{-1}\right) \cup \left( \bigcup_{Q \in \Ass T} Q \cap L_{-1}\right).$$ 
Then $\xi$ is a non-zero divisor on $N'$ and $T$. So $\sigma(\xi)$ is also $N' \oplus T$-regular for all $\sigma \in G$. 
Set $\delta= \prod_{\sigma \in G} \sigma(\xi)$. Then $\delta \in T^G_{-}$ and
$N_i \hookrightarrow N_{i-|\delta|}$ for all $i\leq -m$. Taking invariants  the result follows.
\end{proof}

\section{proof of Theorem \ref{vanish-test}}
\s \label{eclair}{\it Construction} 1: We assume $(A, \m)$ is a \CM \ ring containing a field of characteristic zero. 
Let $G$ be a finite group and $\Phi: G \to Gl_m(A)$ be a group homomorphism.
Let $G$ acts on $R=A[\underline{\bf X}]$ linearly (fixing $A$). 
Choose $f_1, \ldots, f_m \in S_+$  as in Proposition \ref{grade}. 

 Let $\D=D_A(R) = A_m(A)$ be the $m^{th}$ Weyl algebra over $A$. 
 The $G$ action on $R$ can be extended to $\D$ and in particular
 to
 $T=A[\partial_1, \ldots, \partial_m]$ with $\deg \partial_i= -1$. 
 As discussed earlier $G$ acts linearly on $T$ fixing $A$.
 Similarly as Proposition \ref{grade} we can show that there exists homogeneous invariants $g_1, \ldots, g_m$ with $\deg g_i = -|G|!$ such that
 it is  a  homogeneous $T$-regular sequence (and hence a $T^G$-regular sequence). Furthermore
 $\ov{T} = T/(g_1,\ldots, g_m)T$ is a graded free $A$-module with $\ov{T}_j = 0$ for $j < -|G|^m$.
 We call $\underline{f}$ to be { \it a set of  fundamental invariants} of $A$ and $G$ and 
 we call $\underline{g}$ to be {\it a set of dual fundamental invariant} of $A$ and $G$.

Theorem \ref{vanish-test} follows from the following more general result.
\begin{theorem}\label{ref-vanish-test-main}[with notation as in \ref{std}] Further assume $A$ is Cohen-Macaulay.
Set $c = |G|!$.
Set $M= H_{I_1}^{i_1}(H_{I_2}^{i_2}(\cdots (H_{I_l}^{i_l}(S))\cdots))= \bigoplus_{n \in \Z}M_n$ where
$I_1, \ldots, I_l$ are homogeneous ideals in $S$. If $M_i=0$ for $i \in [-m-(c-1)^m, (c-1)^m]$ then $M=0$.
\end{theorem}
\begin{proof}
We may localize at $P$, a prime ideal of $A$. Thus we may assume $A$ is a \CM \  local ring. We do Construction 1 described above.
Set $\D=A_m(A)$ be the $m^{th}$ Weyl algebra over $A$. Now we have $H_i(\underline{\bf X}; N)_j=0$ for $j \neq 0$ and $H_i(\underline{\bf \partial}; N)_j=0$ for $j \neq -m$ where $N=H_{I_1R}^{i_1}(H_{I_2R}^{i_2}(\cdots (H_{I_lR}^{i_l}(R))\cdots))$. Moreover, we have $R/(f)= A \oplus \overline{R_1} \oplus \cdots \oplus \overline{R_d}$ where $\overline{R_i}= A^{l_i}$ with $l_i \geq 0$. Set $W_i= \overline{R_i} \oplus \cdots \oplus \overline{R_d}$ where $d = (c-1)^m$. We have short exact sequences 
\begin{align*}
& 0 \to W_1 \to R/(f) \to A=R/(X_1, \ldots, X_m) \to 0 \quad \text{and} \\
&0 \to W_2 \to W_1 \to R_1^* \to 0.
\end{align*}
Note $R_1^* \cong \left(R/(X_1, \ldots, X_m)\right)^{l_1}(-1)$ as $R$-modules.
Iterating we get \[0 \to W_d \to W_{d-1} \to R_{d-1}^* \to 0\] where $R_{d-1}^* \cong R/(X_1, \ldots, X_m)^{l_{d-1}}(-d+1)$
and $W_d \cong \left(R/(\underline{\bf X})\right)^{l_d}(-d)$ as $R$-modules. From the above short exact sequences
we get that $\Tor_i^R(N, W_d)=0$ for $i \neq d$ and $\Tor_i^r(N, R^*_{d-1})=0$ for $i \neq d-1$. Thus
$\Tor_i^R(N, W_{d-1})=0$ for $i \neq d, d-1$. Iterating we get \[\Tor_i^R(N, R/(f))_j=0 \quad \text{for } j \neq d, d-1, \ldots, 0.\]

Let $T=A[\partial_1, \dots, \partial_m]$ with $\deg \partial_i=-1$.
Set $T/(\underline{g})= \overline{T_{-r}} \oplus \overline{T_{-r+1}} \oplus \cdots \oplus \overline{T}_{-1} \oplus A$
where $r= (c-1)^m$,  and $\overline{T_j} \cong A^{h_j}$. Set $V_i= \overline{T_{-r}}\oplus \cdots \oplus \overline{T_i}$ for $i=0, -1, \ldots, -r$. Then we have short exact sequences of $T$-modules
\begin{align*}
&0 \to V_{-1} \to T/(\underline{g}) \to A=T/(\underline{\partial}) \to 0 \quad \text{and}\\
& 0 \to V_{-2} \to V_{-1} \to T_{-1}^* \to 0
\end{align*} 
where $T_{-1}^* \cong \left(T/(\underline{\partial})\right)^{h_{-1}}(-1)$ as $T$-modules. Iterating we get \[0 \to \overline{T_{-r}} \to V_{-r
+1} \to T_{-r+1}^* \to 0\]where 
$T_{-r+1}^* \cong \left(T/(\underline{\partial})\right)^{h_{-r+1}}(-r+1)$ and $\overline{T_{-r}} \cong \left(T/(\underline{\partial})\right)^{h_{-r}}(-r)$. Since \\  $\Tor_i^T(T/\underline{\partial}, N)_j =0$ for $j \neq -m$ so $\Tor_i^T(N, \overline{T_{-r}})_j=0$ for $j \neq -m-r$ which implies $\Tor_i^T(N, V_{-r+1})_j=0$ for $j \neq -m-r, -m-r+1.$ Iterating we get \[\Tor_i^T(N, T/(\underline{g}))_j=0 \quad \text{for } j \neq -m-r, -m-r+1, \ldots -m.\] 

Thus $H_i(\underline{g}, N)_j =0$ for $j \neq -m-r, \ldots, -m$ where $r = (c-1)^m$ and $H_i(\underline{f}, N)_j =0$ for $j \neq 0, 1, \ldots, (c-1)^m$ for all $i \geq 0$. The Koszul complex $\K(\underline{f}, N)$ is a complex of $R*G$-modules. Clearly $\K(\underline{f}, N)^G= \K(\underline{f}, M)$. Furthermore by \cite[2.8]{TP2} $G$ acts on $H_i(\underline{f}, N)$ and $H_i(\underline{f}, N)^G= H_i(\K(\underline{f}, M))= H_i(\underline{f}, M)$. It follows that $H_i(\underline{f}, M)_j=0$ for all $i \geq 0$ and $j \neq 0, 1, \ldots, (c-1)^m$. Similarly we get that $H_i(\underline{g}, M)_j=0$ for all $i \geq 0$ and $j \neq -m- (c-1)^m, \ldots, -m$.

Suppose $M_j=0$ for all $j \in [-m- (c-1)^m, (c-1)^m]$.
\begin{claim*}
$M_i=0$ for all $i$.
\end{claim*}

{\it Case} 1: Let $i \geq 0$. If $0 \leq i \leq (c-1)^m$ then by our assumption $M_i=0$. So let $i > (c-1)^m$. We have part of the Koszul complex $K(\underline{f}, M)$ \[ \left(M(-2c)\right)^{\binom{m}{2}} \to \left(M(-c)\right)^m \to M \to 0.\] We have $H_0(\underline{f}, M)_j=0$ for $j > (c-1)^m$. So we get exact sequence \[\left(M_{j-c}\right)^m \to M_j \to 0\] for all $j> (c-1)^m$. The result follows.

{\it Case} 2: Let $i<0$. If $-m- (c-1)^m \leq i \leq 0$ then $M_i=0$ by our assumption. Consider the part of the Koszul complex \[M(+c)^m \overset{(g_1, \ldots, g_m)}{\lrt} M \to 0.\] Now we have $H_0(\underline{g}, M)_j=0$ for all $j< -m- (c-1)^m$. Thus we get exact sequence \[\left(M_{j+c}\right)^m \to M_j \to 0\] for $j< -m - (c-1)^m$. The result follows.
\end{proof}

\section{Injective dimension and infinite generation}
\s \label{std-reg}\emph{Construction} 2: \\
Throughout $A$ is a regular ring containing a field $K$ of characteristic zero and let $G$ be a finite group with a 
homomorphism $G \rt GL_n(K)$. 
Let $P$ be a prime ideal of $A$. Set $C=A_P$ and $B=\widehat{A_P}$. We note $B[X_1, \ldots, X_m]^G \cong B \otimes_A (A[\underline{X}]^G)$. By Cohen-structure theorem $B=k(P)[[Y_1, \ldots, Y_d]]$ where $d=\hgt_A P$ and $k(P)$ is the residue field of $B$. We note that  we may assume that $K$ is a subring of $k(P)$; for instance see \cite[Pg.\ 73]{C}.
 Set $T=B[\underline{X}]$. Note $G$ acts on $A_m(T)= T \otimes_{k(P)} A_m(k(P))$. Let $D_B$ be $k(P)$-linear ring of
 differential operators over $B$. Note the action of $G$ on $T$ can be naturally extended to
$D_B[X_1,\ldots, X_m]$. As discussed earlier in \ref{diff-op-local-2} we have a natural $G$-action on $\D= A_m(D_B)$, the $m^{th}$-Weyl algebra over $D_B$ extending the $G$-action on $D_B[X_1,\ldots, X_m]$ fixing $D_B$. If $I_1, \ldots, I_s$ be homogeneous ideals in $T^G$ then set $N=H_{I_1T}^{i_1}(H_{I_2T}^{i_2}(\cdots(H_{I_sT}^{i_s}(T))\cdots))$ and $N^G=M= H_{I_1}^{i_1}(H_{I_2}^{i_2}(\cdots(H_{I_s}^{i_s}(T^G))\cdots))$. We note that it is obvious that $N$ is a $\D$-module and is also $G$-invariant. So $N$ is a $\D*G$-module and $N^G=M$.

\begin{lemma}\label{inj}
Let $B = k[[Y_1,\ldots, Y_d]]$  and let $T = B[X_1,\ldots, X_m]$. Let $G$ be a finite group with a group homomorphism $G \rt GL_m(k)$. We let $G$ act linearly on $T$ and let $T^G$ be the ring of invariants. 
Let $\n$ be maximal ideal of $B$ and let $E_B(k)$ denote the injective hull of $k$ 
as a $B$-module. Let $M = \oplus_{n \in \Z}M_n=\f(T)=H_{I_1}^{i_1}(H_{I_2}^{i_2}(\cdots(H_{I_s}^{i_s}(T^G))\cdots))$
where $I_1,\ldots, I_s$ are graded ideals of $T^G$. If $M_t$ is $\n$-torsion, then $M_t \cong E_B(k)^{s_t}$ where $s_t$
is possibly infinite.
\end{lemma}
\begin{proof}
Set $N= H_{I_1T}^{i_1}(H_{I_2T}^{i_2}(\cdots(H_{I_sT}^{i_s}(T))\cdots))$.
Let $D_B$ be the ring of $k$-linear differential operators of $B$ and let $\D = A_m(D_B)$ be the $m^{th}$-Weyl algebra over $D_B$.
As discussed in \ref{diff-op-local-2} we have a natural $G$-action on $\D= A_m(D_B)$, the $m^{th}$-Weyl algebra over $D_B$ extending the $G$-action on $D_B[X_1,\ldots, X_m]$ fixing $D_B$. 
 Then $N$ is a $\D*G$-module  and $N^G \cong M$ as graded $\D^G$-modules.
 So $M_t$ is a $(\D^G)_0$-module. Note that $(\D^G)_0 \supseteq D_B$. So $M_t$ is a $D_B$-module annihilated by $\n$.
 Hence $M_t \cong E_B(k)^{s_t}$ for some ordinal possibly infinite see \cite[2.4(a)]{Lyu-1}.
\end{proof}

We need the following Lemma from \cite[1.4]{Lyu-1}.
\begin{lemma}\label{ilyu}
Let $A$ be a Noetherian ring and let $M$ be an $A$-module {\rm(}$M$ need not be finitely generated{\rm)}. Let $P$ be a prime ideal in $A$. If $(H_P^j(M))_P$ is injective $A$-module for all $j \geq 0$ then $\mu_j(P, M)= \mu_0(P, H_P^j(M))$.
\end{lemma}

The following result is basic to all further results.
\begin{theorem}\label{L-P}
Let $A$ be a regular ring containing a field $K$ of characteristic zero. Let $G$ be a finite
group and $\Phi: G \to Gl_m(K)$ be a group homomorphism. Let $G$ acts on $R=A[\underline{\bf X}]$ linearly (fixing $A$).
Let $S$ be ring of invariants of $G$. Set $M=\f(T)=H_{I_1}^{i_1}(H_{I_2}^{i_2}(\cdots(H_{I_s}^{i_s}(T^G))\cdots))= \bigoplus_{n \in \Z}M_n$. Then for any $P \in \Spec A$ 
\[
\mu_j(P, M_t)= \mu_0(P, H_P^j(M_t)).\]
\end{theorem}

\begin{proof}
Suffices to show that $(H_P^j(M_t))_P$ is an injective $A$-module. Set $V= H_{PS}^j(M)$ and $V_t = H_P^j(M_t)$. We note that
$(V_t)_P= V_t \otimes_A \widehat{A_P}$. Since $V_t$ is $P$-torsion so we have $V_t = (V \otimes_A \widehat{A_P})_t = 
\left(H_{PS_P}^j(M_P)\right)_t$. By Lemma \ref{inj} we have $(V_t)_P= E_{\widehat{A_P}}(k(P))^s = E_{A_P}(k(P))^s = E_A(A/P)^s$ where $s$ is an ordinal number, possibly infinite. 
The result follows from Lemma \ref{ilyu}.
\end{proof}

As an application of Theorem \ref{L-P}    we have:
\begin{corollary} \label{injdim-supp}
Let $t \in \Z$. Then $\injdim M_t \leq \dim_A M_t$.
\end{corollary}
\begin{proof}
For any prime ideal $P$ of $A$ we have \[\mu_j(P, M_t)= \mu_0(P, H^j_P(M_t)).\] Moreover, by Grothendieck's Vanishing Theorem 
$H^j_P(M_t)= 0$ for all $j> \dim M_t$. So $\mu_j(P, M_t)=0$ for all $j> \dim M_t$.
\end{proof}

Another  use of our Construction \ref{std-reg} is the following 
\begin{theorem}\label{inf-gen-proof}(with hypotheses as in \ref{std-reg}). Further assume $A$ is a domain.
Let $I$ be a homogeneous ideal in $S$. Assume $I \cap A \neq 0$. If $H^i_I(S)_t \neq 0$ then
$H^i_I(S)_t$ is NOT finitely generated as an $A$-module.
\end{theorem}
\begin{proof}
Set $Q = I \cap A$. Suppose if possible $ L = H^i_I(S)_t$ is a non-zero finitely generated $A$-module.

 Let $\m$ be a maximal prime  ideal in $A$ belonging to the support of $L$. Notice as $L$ is $Q$-torsion we have $\m \supseteq Q$. Let $B = \widehat{A_\m}$ the completion of $A$ \wrt \ $\m$. We note that the image of $Q$ in $B$ is non-zero.
 Set $J = QB$.
 
 We now do the construction as in  \ref{std-reg}.
 By Cohen-structure theorem $B = K[[Y_1,\ldots, Y_g]]$ where $K = A/\m $  and $g = \hgt_A \m$.
 Let $D_B $ be the ring of $K$-linear differential operators on $B$ and set $\D = A_m(D_B)$, the $m^{th}$-Weyl algebra
 over $\Lambda$. Then by  \ref{std-reg} we get that $M = H^i_{IB}(S) $ is a graded 
 $\D^G$-module.  In particular $V = M_c$ is a $D_B$-module (since $D_B$ is a subring of $\D^G_0$).
 
 Let $V$ be generated as a $R$-module by $v_1,\cdots, v_l$. Each $v_j$ is killed by a power of $J$. It follows that there exists $n$ such that $J^n V = 0$. 
 Let $\n$ be the maximal ideal of $B$. Choose $p \in \n$ of smallest $\n$-order $s$ such that $p V = 0$. We note that for all $i = 1, \ldots, d$,
 $$\delta_i p = p \delta_i  + \delta_i(p),$$
 holds in $\Lambda$.  Notice $p \delta_i V = 0$ since $\delta_i V \subseteq V$ (as $V$ is a $D_B$-module).  So $\delta_i(p) V = 0$.
 
 We note that if $s \geq 1$ then some $\delta_i(p)$ will have $\n$-order less than $s$. It follows $s = 0$. Thus $p$ is a unit. So $V = 0$, a contradiction as we were assuming $V$ to be non-zero.
 Thus our assumption is incorrect. Therefore $ H^i_I(R)_t$ is NOT finitely generated as an $A$-module.
\end{proof}
\section{Bass numbers}
\s Let $A = k[[Y_1,\ldots, Y_d]]$ and let $\m$ be its maximal ideal. 
Let $G$ be a finite group with a group homomorphism $G \rt GL_m(k)$. Let $G$
act on $R=A[\underline{\bf X}]$ linearly (fixing $A$). Let $S=R^G$ be ring of invariants of $G$. 
Set $U =  R/\m R= k[\underline{\bf{X}}]$. Clearly $G$ acts on $U$. 
Let $N$ be an $R*G$-module. Then $N/\m N$ is a $U*G$-module. 

\begin{lemma}\label{invar-mod}
$\left(N/\m N\right)^G \cong N^G/\m N^G$. 
\end{lemma}
\begin{proof}
We have a split short exact sequence of $R^G$-modules
\[0 \to N^G \to N \to N/N^G \to 0.\] Applying $- \otimes_{R^G}R^G/\m R^G$ gives a split exact sequence \[0 \to N^G/\m N^G \overset{i}{\hookrightarrow} N/\m \to * \to 0.\] Let $\xi \in \left(N/\m N\right)^G$. Then $\xi= \overline{u}$ where $u \in N$. Now $\sigma(\xi)=\xi$ implies $\sigma(u)= u+ \sum_{i=1}^d a_i n_i(\sigma)$. Thus $\rho(u)= 1/|G| \left(\sum \sigma(u)\right)= u + \sum_{i=1}^d a_i \sum_{\sigma \in G} n_i(\sigma)$. Hence $\rho(u) \in N^G$ and $i(\overline{\rho(u)})= \overline{u}= \xi$ in $N/\m N$. So the result follows.
\end{proof} 

\s Recall $T = A[\partial_1,\ldots,\partial_m]$ with $\deg \partial_i = -1$.
We recall that $G$ acts linearly on $T$ fixing $A$.
Set $V = T/\m T = k[\partial_1,\ldots,\partial_m]$.
Furthermore if $N$ is an $T*G$-module then $N/\m N$ is a $V*G$-module and $\left(N/\m N\right)^G \cong N^G/\m N^G$. 

\s \label{const-bass-rachel} {\it Construction 3}: Let $A=k[[Y_1, \ldots, Y_d]]$ and $\n=(Y_1, \ldots, Y_d)$ be unique maximal ideal of $A$. Let $R=A[X_1, \ldots, X_m]$ be standard graded with $\deg A=0$ and $\deg X_i=1$. Let 
$G$ be a finite group with a group homomorphism 
$G \rt Gl_m(k)$. Let $G$ acts linearly on $R$ fixing $A$. Let $S = R^G$ be ring of invariants of $G$. 
Let $I_1,\ldots, I_s$ be homogeneous ideals in $S$. Set $\f = H^{i_1}_{I_1} \circ  H^{i_2}_{I_2} \circ \cdots \circ H^{i_s}_{I_s}$ and set $\f' = H^{i_1}_{I_1R} \circ  H^{i_2}_{I_2R} \circ \cdots \circ H^{i_s}_{I_sR}.$  
Set $M^*= \f(S), M=H_{\n R^G}^j\left(\f(S)\right), N^*= \f'(R), N=H_{\n R}^j\left(\f'(R)\right)$ such that $(N^*)^G \cong M^*$. Then $N$ is an $R*G$-module and $N^G \cong M$. 

Let $D_{A/k}$ be ring of $k$-linear differential operator over $k$. Set $\D_R= A_m(D_{A/k})$.
Then $G$ acts on $\D_R$ (fixing $D_{A/k}$) see \ref{diff-op-local-2}. Set $\D_R^G$ the ring of invariants of $D_k$.
Then $N$ is a $\D_R *G$-module. Set $A'=k[[Y_1, \ldots, Y_{d-1}]]$ considered as a sub-ring of $A$ and 
$R'=A'[X_1, \ldots, X_m]$. Then $D_{A'/k}\subset D_{A/k}$ and $\D_{R'}\subseteq \D_R$.

Now the map $N \overset{Y_d}{\lrt} N$ is a $\D_{R'}*G$-linear (as $\sigma(Y_d) = Y_d$ for all $\sigma \in G$). 
So $H_l(Y_d, N)$ are $D_{R'}*G$-modules for $l = 0,1$. Thus we have a short exact sequence of $D_{R'}*G$-modules
\[0 \to H_1(Y_d, N) \to N \overset{Y_d}{\lrt} N \to H_0(Y_d, N) \to 0.\] Taking invariants we get
$H_l(Y_d, M) \cong H_l(Y_d, N)^G$ for $l=0, 1$. Iterating we get that $H_l(\underline{Y}, N)$ is an $A_m(k)*G$-module
and $H_l(\underline{Y}, N)^G \cong H_l(\underline{Y}, M)$ where $\underline{Y}= Y_1, \ldots, Y_d$. We note that 
$N_i \cong E_A(k)^{s_i}$ and $M_j \cong E_A(k)^{d_j}$ for some $s_i, d_j$ (possibly infinite). So
$H_d(\underline{Y}, N)_i \cong k^{s_i}$ and $H_d(\underline{Y}, M)_j \cong k^{d_j}$. 
We also note that by Theorem \ref{L-P} we get $\mu_j(\n, M^*_n) = \mu_0(\n, M_n)$.

\begin{remark}
$H_d(\underline{Y}, N)$ is also a holonomic $A_m(k)$-module. So $H_i(\underline{X}, \widetilde{N})$ and $H_i(\underline{\partial}, \widetilde{N})$ are finite dimensional $k$-vector space for all $i$ where $\widetilde{N}= H_d(\underline{Y}, N)$. 
As $N$ is a generalized Eulerian $\D_R$-module we get that 
 $\widetilde{N}$ is a generalized Eulerian $A_m(k)$-module; see \cite[4.26]{TP2}.  So  by \cite[Theorem 1]{P2}, \cite[5.1]{PS}
\[H_i(\underline{X}, \widetilde{N})_j= 
 \left\{
 	\begin{array}{ll}
 		\text{finite dimensional } k-\text{vector space } & \mbox{if } j=0 \\
 	0 & \mbox{if }j \neq 0
 	\end{array}
 \right.\] and 
 \[H_i(\underline{\partial}, \widetilde{N})_j= 
  \left\{
  	\begin{array}{ll}
  		\text{finite dimensional } k-\text{vector space } & \mbox{if } j=-m \\
  	0 & \mbox{if }j \neq -m
  	\end{array}
  \right.\]
\end{remark}
 
\begin{lemma}
Set $c = |G|!$.
Let $f_1, \ldots, f_m \in R^G$ (and $g_1,\ldots. g_m \in T^G$ ) be a set of  fundamental (and dual) invariants of
$G$ \wrt \ $A$ (see \ref{eclair}) with $\deg f_i = c$ and $\deg g_i = -c$ for all $i$.
Set $U = R/\m R = k[X_1,\ldots, X_m]$  $L = T/ \m T = k[\partial_1,\ldots, \partial_m]$ and $D = A_m(k)$. Consider the natural linear action of $G$ on $U$,$V$ and $D$. Let
$E$ be a graded holonomic generalized Eulerian $D$-module on which $G$ acts.
Then
\begin{enumerate}[\rm (1)]
\item
$\dim_k H_i(\underline{f}, E)< \infty$.
\item
$\dim_k H_i(\underline{g}, E)< \infty$.
\item
$\dim_k H_i(\underline{f}, E^G)< \infty$.
\item
$\dim_k H_i(\underline{g}, E^G)< \infty$.
\end{enumerate}
\end{lemma}
\begin{proof}
(1) Let $U/(f)= C_0 \oplus C_1 \oplus \cdots \oplus C_{c^m}$ where
$\dim_k C_i < \infty$. Set $W_i =\langle C_i, \ldots, C_{c^m}\rangle$. Then $W_0= U/(f)$ and we have short
exact sequences \[0 \to W_{i-1} \to W_i \to \overline{C_i} \to 0\] where $C_i \cong k^{\dim_k V_i}(-i) \cong \left(T/\underline{X}T\right)^{\dim_k V_i}(-i)$ for all
 $0 \leq i \leq c^m$. By argument similar to Theorem \ref{ref-vanish-test-main} we get \[\Tor_i^R(R/(\underline{f}), \widetilde{N})_j \cong H_i(\underline{f}, \widetilde{N})_j=0 \quad \text{for } j \neq 0, 1, \ldots, c^m.\] Also the filtration yields that $\dim_k H_i(\underline{f}, \widetilde{N})_j$ is finite dimensional for all $j \in \Z$. Thus $\dim_k H_i(\underline{f}, \widetilde{N}) < \infty$.

(2) This is similar to (1).

(3) The Koszul complex $K(\underline{f}, E)$ is $G$-invariant as $\underline{f} \in U^G$. By \cite[2.8]{TWP}  we get that 
\[
H_i(\underline{f}, E^G) \cong H_i(\underline{f}, E)^G.
\]
The result follows.

(4) This is similar to (3).
\end{proof}

\s \label{canteen} \emph{Notation}:  Set $c = |G|!$. Degree of each term in a fundamental sequence is $c$ and dual fundamental sequence is $-c$. Set $\Theta_l = l +c \Z$ for $0 \leq l \leq c-1$. 
Throughout the functor $\f $ will denote 
$H_{I_1}^{i_1} \circ H_{I_2}^{i_2} \circ \cdots \circ H_{I_s}^{i_s}$ for some homogeneous ideals $I_1,\ldots, I_s$ in $S = R^G$. Set $\f' = H_{I_1R}^{i_1}\circ H_{I_2R}^{i_2} \circ \cdots \circ H_{I_sR}^{i_s}$
\begin{theorem}\label{bass-finite}(with hypotheses as in \ref{canteen}). Let $P$
be a prime ideal in $A$.
If $\mu_j(P, \f(S)_{n_0})< \infty$ for some $n_0 \in \Theta_l$  then $\mu_j(P, \f(S)_n) < \infty$ for all $n \in \Theta_l$.
\end{theorem}

\begin{proof}
We have $\mu_j(P, \f(S)_i)= \mu_0(P, H_P^j(\f(S)_i))$. Set $M= H_{PS}^j(\f(S))$ and $N= H_{PR}^j(\f'(R))$. Clearly $N^G=M$. Localize at $P$ and complete. So we may assume that $(A, \m)$ complete and $P=\m$. 
Let $f_1, \ldots, f_m \in R^G$ (and $g_1,\ldots. g_m \in T^G$ ) be a set of  fundamental (and dual) invariants of $G$ \wrt \ $A$.  
 Then by the procedure described above \ref{const-bass-rachel} we may reduce to the case $A=k$. Notice $M_j^* \cong k^{s_j}$ where $s_j= \mu_j(P, M)$ where $s_j$ is some ordinal. 

We have $H_i( \underline{f}, M^*)$ is finite dimensional $k$-vector space for all $i$.  Consider the strand of the Koszul complex $\K(\underline{f}, M^*)$,\[(M_{n_0}^*)^m \to M_{n_0+c}^* \to 0\] Thus we get the short exact sequence \[(M_{n_0}^*)^m \to M_{n_0+c}^* \to H_0(\underline{f}, M^*)_{n_0+c} \to 0.\] Thus $\dim_k M_{n_0+c}^* < \infty$. Iterating we get $\dim_k M_{n_0+tc}^*< \infty$ for all $t \geq 0$. We now look at the dual fundamental invariants of $G$. We have $\dim_k H_i(\underline{g}, M^*)< \infty$. Moreover, we have the short exact sequence \[(M_{n}^*)^m \to M_{n-c}^* \to H_0(\underline{g}, M^*)_{n-c} \to 0.\]  Thus an easy induction yields $\dim_k M_{n_0+tc}^*< \infty$ for all $t \leq 0$. So $\dim_k M^*_n< \infty$ for all $n \in \Theta_l$. The result follows.
\end{proof}

\begin{theorem}\label{basspoly}
If $\mu_j(P, \f(S)_n)< \infty$ for all $n \in \Theta_l$, then there exists $\alpha(X)$, $\beta(X) \in Q[X]$ of degree $\leq m-1$ such that 
\begin{enumerate}[\rm (1)]
\item $\mu_j(P, \f(S)_n)=\alpha(t)$ for all $t \gg 0$ where $n = l + ct$.
\item $\mu_j(P, \f(S)_n)= \beta(t)$ for all $t \ll 0$ where $n = l + ct$.
\end{enumerate}
\end{theorem}
\begin{proof}
By our construction \ref{const-bass-rachel} we may reduce to $M^*$ a $A_m(k)^G$-module such that $\dim_k M_n^*= \mu_j(P, \f(S)_n)$ for all $n \in \Z$. Now $\K(\underline{f}, M^*)$ is a complex with finite dimensional homology. Note that $\deg f_i= c$ for all $i$. We have \[\K(\underline{f}, M^*): 0 \to M^*_{n-dc} \to \left(M^*_{n-(d-1)c}\right)^{\binom{m}{m-1}} \to \cdots \to \left(M^*_{n-c}\right)^{\binom{m}{1}} \to M_n^* \to 0.\] By taking $n \in \Theta_i$ we get the function $\psi(t)= \dim_k M_{n+ct}^*$ for $t \geq 0$. Then we get $\Delta^m \psi(t)$ vanish for all but finitely many $t \geq 0$. it follows that $\psi(t)$ is a polynomial type of degree $\leq m-1$. Thus $(1)$ follows.

$(2)$ follows by considering dual invariants and a similar proof.
\end{proof}

\section{Associate primes}
Let $E= \bigoplus_{n \in \Z}E_n$ be a graded module over a graded ring $S= \bigoplus_{n \geq 0}S_n$ (not necessarily standaed graded). For associated primes we ask 
\begin{enumerate}[\rm 1)]
\item Is $\bigcup_{n \in \Z}\Ass_{S_0} E_n$ is finite?
\item Does there exists $v>0$  and $n_0, n_0'$ such that if $\Theta_l = l + v \Z$ for some $0 \leq i \leq v-1$ then 
\begin{enumerate}[\rm i)]
\item $\Ass_{S_0} E_n = \Ass_{S_0} E_{n_0}$ for all $n \in \Theta_l$ and $n \geq n_0$?
\item $\Ass_{S_0} E_n = \Ass_{S_0} E_{n'_0}$ for all $n \in \Theta_l$ and $n \leq n'_0$?
\end{enumerate}
\end{enumerate}
We call $(2)$; {\it periodic stability} of associated primes of graded components  of $E$ \wrt \ $S_0$.
The following result implies Theorem \ref{ass-primes}.
\begin{theorem}\label{ass-primes-sect}
Assume either $(A, \m)$ is a regular local ring or is a smooth affine variety. Let $R=A[\underline{X}]$ be standard graded with $\deg A=0$ and $\deg X_i=1$ for all $i$. Let $S$ be the ring of invariants of $G$. Then for any $\f= H_{I_1}^{i_1} \circ \cdots \circ H_{I_s}^{i_s}$
\begin{enumerate}[\rm 1)]
\item $\bigcup_{i \in \Z} \Ass_A \f(S)_i$ is a finite set.
\item There exists periodic stability of associated primes of  graded components of $\f(S)$.
\end{enumerate}
\end{theorem}
\begin{proof}
1)
Set $\f ' = H_{I_1R}^{i_1}\circ H_{I_2R}^{i_2} \circ \cdots \circ H_{I_sR}^{i_s}$
 We have the split exact sequence 
\[
\xymatrix{
0 \ar[r]& \f(S) \ar@{^{(}->}[r]^i & \f'(R)  \ar@/^1pc/[l]^\rho \ar[r] & \f'(R)/\f(S) \ar[r] & 0}\]
of $S=R^G$-modules. Then $\f(S)_n$ is an $A$-summand of $\f'(R)_n$. By our earlier result
 $\bigcup_{n \in \Z} \Ass_A \f'(R)_n$ is a finite set \cite[12.4]{TP2}. 
 So $\bigcup_{n \in \Z} \Ass_A \f(S)_n$ is also a finite set.
 
 2) Set $c = |G|!$. Then $c, -c$ is the common degree of \emph{a} both fundamental and dual invariants respectively of
 $G$ \wrt $A$. Set $\Theta_l = l +c \Z$ where $0 \leq j \leq c-1$. Fix $j$. Let $P \in \bigcup_{n \in \Z} \Ass_A \f(S)_n$. 
 Then $P \in \Ass_A \f(S)_n$ if and only if $\mu_0(P, \f(S)_n)> 0$. Let $n \in \Theta_l$. Then $n= l +ct$ for some $t \in \Z$. Now by Proposition \ref{basspoly} we know that 
either $\mu_0(P, \f(S)_n)$ is infinite for all $n \in \Theta_l$ or
there exits $\alpha(Z), \beta(Z) \in \QQ[Z]$ such that 
\[\mu_0(P, \f(S)_n)= 
  \left\{
  	\begin{array}{ll}
  		\alpha(t) & \mbox{if } t \gg 0 \\
  	\beta(t) & \mbox{if } t \ll 0
  	\end{array}
  \right.\]
The result follows.
\end{proof}

\section*{Acknowledgements}
The author thanks Ms. Sudeshna Roy for help in typing this paper and for many useful comments.

\end{document}